\documentclass[11pt, a4paper,reqno]{amsart}
\usepackage{a4,amsmath,amssymb,epic}
\usepackage{graphicx}
\usepackage{multicol}
\usepackage{rotating}
\usepackage{lscape}
\usepackage{amssymb,amsthm}
\usepackage{enumitem}
\usepackage[all]{xy}
\usepackage{tikz}
\usepackage{young}

\usepackage[vcentermath,enableskew]{youngtab}
\newtheorem{theorem}{Theorem}[section]
\newtheorem{proposition}[theorem]{Proposition}
\newtheorem{lemma}[theorem]{Lemma}

\newtheorem{corollary}[theorem]{Corollary}
\newtheorem{remark}[theorem]{Remark}
\usepackage{tikz}

\title[Number of permutations with same peak set]{Number of permutations with same peak set for signed permutations}
\author[Castro-Velez]{Francis Castro-Velez}
\address{183 Memorial Drive\\
Cambridge, MA 02139}
\email{fncv@mit.edu  }
\author[Diaz-Lopez]{Alexander Diaz-Lopez}
\address{Department of Mathematics\\
University of Notre Dame\\
Notre Dame, IN 46556}
\email{adiaz4@nd.edu}
\author[Orellana]{Rosa Orellana}
\address{Department of Mathematics,
Dartmouth College,
6188 Kemeny Hall, 
Hanover, NH 03755, USA } 
\email{Rosa.C.Orellana@dartmouth.edu}
\author[Pastrana]{Jos\'e Pastrana}
\address{Department of Mathematics\\
University of Puerto Rico\\
 Barbosa Avenue\\
 San Juan, Puerto Rico 00931}
\email{pastrana.c.jose@gmail.com}
\author[Zevallos]{Rita Zevallos}
\address{Swarthmore College\\
500 College Avenue\\
Swarthmore, PA 19081}
\email{rzevall1@swarthmore.edu}

\begin{document}
\maketitle


\begin{abstract}

A signed permutation $\pi =  \pi_1\pi_2 \ldots \pi_n$  in the hyperoctahedral group $B_n$ is a word such that each $\pi_i \in \{-n, \ldots, -1, 1, \ldots, n\}$ and $\{|\pi_1|, |\pi_2|, \ldots, |\pi_n|\} = \left\{1,2,\ldots,n\right\}$. An index $i$ is a peak of $\pi$ if 
 $\pi_{i-1}<\pi_i>\pi_{i+1}$ and  $P_B(\pi)$  denotes the set of all peaks of $\pi$.  Given any set $S$, we define $P_B(S,n)$ to be the set of signed permutations $\pi \in B_n$ with $P_B(\pi) = S$.   In this paper we are interested in the cardinality of the set $P_B(S,n)$.     In \cite{sagan}, Billey, Burdzy and Sagan investigated the analogous problem for permutations in the symmetric group, $S_n$.   In this paper we extend their results to the hyperoctahedral group; in particular  we show that $\#P_B(S,n) = p(n)2^{2n-|S|-1}$ where $p(n)$ is the same polynomial found in \cite{sagan} which leads to the explicit computation of interesting special cases of the polynomial $p(n)$.  In addition we have extended these results to the case where we add $\pi_0=0$ at the beginning of the permutations, which gives rise to the possibility of a peak at position 1,  for both the symmetric and the hyperoctahedral groups.   
 

\end{abstract}

\section{Introduction}

A permutation $\pi=\pi_1\pi_2\ldots \pi_n$ in the symmetric group $S_n$ has a peak at index $i$ if 
$\pi_{i-1}<\pi_i>\pi_{i+1}$.   The peak set of $\pi$ is defined to be $P(\pi) = \{i\, |\, \mbox{ $i$ is a peak of $\pi$}\}$, then we define 
\[P(S,n) = \{ \pi \in S_n \, |\,  P(\pi) = S\}\]  to be the set of all permutations with the same peak set $S$.   For example, the permutation $\pi = 2\ 6 \ 5 \ 1 \ 4\ 3$ has peaks at position 2 and 5, hence $P(\pi, 6) = \{ 2,5\}$. 

Stembridge \cite{stembridge} was one of the first to study the combinatorics of peaks, in particular he gave a peak analog of Stanley's theory of poset partitions.  Additional interest in the study of peaks arose when Nyman \cite{nyman}  showed that summing permutations according to their peak sets leads to a non-unital subalgebra of the group algebra of the symmetric group.   

In a recent paper \cite{sagan} Billey, Burdzy and Sagan considered the cardinalities of the sets $P(S,n)$.  They discovered that $\#P(S,n) = p(n) 2^{n-\#S-1}$  for some polynomial $p(n)$ depending on $S$; they also computed special cases of the polynomial $p(n)$.  One motivation for studying peaks of permutations  lies in probability theory, in a recent paper Billey, Burdzy, Pal and Sagan \cite{billeysagan2} studied distributions on graphs that are related to random permutations with certain peak sets.   Besides the applicability to probability theory, the problem of enumerating permutations and signed permutations with respect to a given statistic is an interesting problem on its own, for example the enumeration of permutations related to peak sets has also been studied in \cite{kasraoui, strehl, warren}. 

It is natural that when a result related to the symmetric group (the Coxeter group of type A) is obtained one wishes to generalize it to other Coxeter groups.  In this paper we generalize the results in \cite{sagan} to the group of signed permutations, the hyperoctahedral group $B_n$ (the Coxeter group of type B).  A signed permutation  $\pi =  \pi_1\pi_2 \ldots \pi_n$ is a word such that each $\pi_i \in \{-n, \ldots, -1, 1, \ldots, n\}$ and $\{|\pi_1|, |\pi_2|, \ldots, |\pi_n|\} = \left\{1,2,\dots,n\right\}$.  A peak of a signed permutation is defined in exactly the same way as for regular permutations.   We will denote by $P_B(\pi)$ the set of peaks of a  signed permutation $\pi$ and define
\[P_B(S,n) = \{ \pi \in B_n\, |\, P_B(\pi) = S\}.\] 
 We show that $\#P_B(S,n) = 2^{2n-\#S-1} p_B(n)$, where $p_B(n)$ is the same polynomial as for the symmetric group, this generalizes Theorem 3 in \cite{sagan}.  We also consider special cases of the polynomials $p_B(n)$.     

Peaks for signed permutations are also of interest in the construction of algebraic structures.  In \cite{bergeron} Bergeron and Hohlweg have described peak analogues of the peak algebras  for the hyperoctahedral group and Petersen \cite{petersen} considered peak algebras of the hyperoctahedral group when the signed permutations are grouped by number of peaks. 

The second part of our paper considers the enumeration of peak classes when we put a zero at the beginning of the permutations for both the symmetric and hyperoctahedral groups.  That is, we consider permutations of the form $\pi_0\pi_1\cdots \pi_n$ where $\pi_0=0$. These permutations arose in the study of unital peak algebras of the symmetric group \cite{aguiarbergeronnyman, aguiarnymanorellana}.   In the case of the symmetric group, adding a zero at the beginning of every permutation has the effect of having the identity as the unique permutation with no peaks.  We denote by $\widehat{P}(S,n)$ the set of permutations with a zero added in $S_n$ with peak set $S$ and $\widehat{P}_B(S,n)$ the corresponding set for the hyperoctahedral group.   We generalize results obtained in \cite{sagan} to $\widehat{P}(S,n)$ and $\widehat{P}_B(S,n)$.  In particular, we give a method for computing $\#\widehat{P}(S, n)$ and $\#\widehat{P}_B(S,n)$ and compute these numbers for special sets $S$.

We now give a more detail description of the contents of this paper.  In Section \ref{Bn}, we prove that $\#P_B(S,n)=p_B(n)2^{2n-s-1}$ where $p_B(n)$ is an integral polynomial in terms of $n$. In addition, we show that the polynomials $p_B(n)$ in $B_n$ are equal to the polynomials $p(n)$ in $S_n$ found in \cite{sagan}.   We also show that the values for $\#P_B(S,n)$ are symmetric for a fixed $n$ as we vary the set $S$.   In Section \ref{B0n},  we provide a method to compute $\#\widehat{P}_B(S,n)$ for any $S$. We find that if $S=\emptyset$, $\#\widehat{P}_B(S,n)$ can be written in terms of the Stirling numbers of the second kind. Another result in this section gives us the parity of $\#\widehat{P}_B(S,n)$. Additionally, we find that $\#P_B(S,n)$ can be written as the sum of $\#\widehat{P}_B(S,n)$ and $\#\widehat{P}_B(S \cup \{1\},n)$ which gives us the results of the previous section when $2 \in S$. Finally, we calculate $\#\widehat{P}_B(S,n)$ for various specific sets $S$. 

Finally, in Section \ref{S0n}, we focus on the symmetric group,  we provide a method to compute $\#\widehat{P}(S,n)$ for any $S$. We also find that $\#P(S,n)$ can be written as the sum of $\#\widehat{P}(S,n)$ and $\#\widehat{P}(S \cup \{1\},n)$ which gives us the results in \cite{sagan} when $2 \in S$. Another result in this section gives us the parity of $\#\widehat{P}(S,n)$. Finally, we calculate $\#\widehat{P}(S,n)$ for various special cases of the set $S$.


\section{Signed Permutations in $B_n$}\label{Bn}

Let $B_n$ be the hyperoctahedral group, i.e.,  the group of signed permutations, and let $\pi = \pi_1 \pi_2 \ldots \pi_n$ be a permutation in $B_n$. Recall that we define a position $i \in \left\{2, \ldots ,n-1\right\}$ as a peak if $\pi_{i-1} < \pi_i > \pi_{i+1}$, and the set $P_B(\pi) $ as the set of all peaks of $\pi$. 

Define a set $S = \left\{i_1 < \dots <i_s \right\}$ to be $n$-admissible if $\#P_B(S,n) \not= 0$. Note that we insist the elements be listed in increasing order.  Notice that $S$ cannot contain two consecutive integers and $S$ is a subset of $\{2, \ldots, n-1\}$.  The minimum possible value of $n$ for which $S$ is $n$-admissible is $i_s+1$, and in that case $S$ is $n$-admissible for all $n \geq i_s+1$. If we make a statement about an admissible set $S$, we mean that $S$ is $n$-admissible for some $n$ and the statement holds for every $n$ such that $S$ is $n$-admissible. It is well-known that the number of $n$-admissible sets is the $n$-th Fibonacci number.  We include a proof for completeness.

\begin{proposition}\label{thm-nadmissible0}
Let $A_n$ be the set of $n$-admissible peak sets $S$. Then the size of $A_n$ is given by the $n$-th Fibonacci number. 
\end{proposition}

\begin{proof}
Note that $A_1 = A_2 = \{\emptyset\}$, thus the result holds for $n=1$ and $n=2$.

Now consider $n \geq 3$. We can write $A_n$ as a union of disjoint sets $A_\alpha$ and $A_\beta$ where $A_\alpha$ is the set of $n$-admissible sets that do not contain $n-1$, and $A_\beta$ is the set of $n$-admissible sets that do contain $n-1$. But note that since $A_{n-1}$ contains all $(n-1)$-admissible peak sets $S$, which cannot contain the element $n-1$, it must be equal to $A_\alpha$. Also, adding $n-1$ to all the peak sets in $A_{n-2}$ gives us $A_\beta$.  Therefore $|A_n| = |A_{n-1}| + |A_{n-2}|$.

\end{proof}


If we fix $n$ and the cardinality of the set $S$, then there exists a set $T$ of the same cardinality as $S$ such that $\#P_B(S,n) = \#P_B(T,n)$.  We make this symmetry property more explicit in the following  proposition. 

\begin{proposition}
Let $S = \left\{i_1,i_2,\ldots,i_k\right\}$ and $T = \left\{n+1-i_k,\ldots ,n+1-i_1\right\}$. Then $\#P_B(S,n) = \#P_B(T,n)$.
\end{proposition}

\begin{proof}

Define $f: B_n \rightarrow B_n$ such that if $\pi = \pi_1\pi_2 \cdots \pi_n$, $f(\pi) = \pi_n \cdots \pi_2\pi_1$. 
Note that $f$ is an involution, because $f(f(\pi)) = \pi$.

Now let $\rho=\rho_1\rho_2\hdots\rho_n \in P_B(S,n)$. If $j$ is a peak of $\rho$ then $n+1-j$ is a peak of $f(\rho)$, hence the peak set of $f(\rho)$ is $\left\{n+1-i_k,\ldots,n+1-i_2,n+1-i_1\right\}$. Thus $f(\rho) \in P_B(T,n)$.  Similarly, if $f(\pi)\in P_B(T,n)$ we can show that $f(f(\pi))\in P_B(S,n)$.  Therefore $\#P_B(S,n) = \#P_B(T,n)$.
\end{proof}

\noindent{\bf Remark: } Note that since $S_n \subseteq B_n$, this result holds in $S_n$ as well.

\medskip

We now prove the following special case which will be the base case for our induction later on.  In what follows we set $[n] = \{ 1, 2, \ldots, n\}$. 
\begin{proposition}\label{prop:nopeakssigned}
The number of signed permutations with no peaks is equal to $2^{2n-1}$.
\end{proposition}

\begin{proof}
Let $\pi=\pi_1 \pi_2 \hdots \pi_n$ be a signed permutation, and $k=\min(\pi_1, \pi_2, \hdots, \pi_n)$.  Since $\{|\pi_1|, |\pi_2|, \ldots, |\pi_n|\} =[n]$ then there is a $\pi_j$ with $|\pi_j|=1$.  Hence $k=1$ or $k<0$. 


Thus the number of signed permutations with no peaks can be divided into those with $k=1$ and those with $k<0$. Let us first assume that $k=1$. Then, we can write $\pi = k_11k_2$ where $k_1$ denotes the portion to the left of $1$ and $k_2$ denotes the portion to the right of $1$. To have $P_B(\pi)=\emptyset$, $k_1$ must be decreasing and $k_2$ must be increasing. We notice that the size of $k_1$ ranges from $0$ to $n-1$ and the values of $\pi_i$ in $k_1$ and $k_2$ range from $2$ to $n$ because $k=1$. This means that the number of signed permutations with no peaks when $k=1$ is equal to the number choices of a subset of elements from $\{2,\ldots, n\}$ to be in $k_1$ since after the choices are made, the rest of $\pi$ is determined. This implies that the number of signed permutations with no peaks and $k=1$ equals to \[\sum_{j=0}^{n-1} \binom{n-1}{j}=2^{n-1}.\] 

Now let $k<0$. Note that the numbers that can go in $\pi$ are from the set $A=\{k+1, k+2, \hdots, -1, 1, \hdots, n \}$. We also notice that the size of $k_1$ can vary from $0$ to $n-1$ again. Although we have more than $n-1$ options available to put in $k_1$, we  note that if an integer $m \in [n]$ is in the permutation, then $-m$ cannot be, and vice versa. As a result we have $n-1$ numbers to put in $k_1$ and then we multiply by $2$ for every element in $A$ for which its negative is in $A$ as well. Furthermore the number of signed permutations with no peaks and $k$ fixed, where $k<0$ equals $2^{n-1} \cdot 2^{k-1}$ (since we are multiplying by 2 for every element in $\{k+1, k+2, \hdots -1\}$). Hence the total number of signed permutations with no peaks and $k<0$ is 
\[\sum_{k=1}^{n} 2^{n-1} \cdot 2^{k-1}=2^{n-1} \sum_{k=1}^{n} 2^{k-1}=2^{n-1}(2^n-1)=2^{2n-1}-2^{n-1}.\] 
But we have to add in what we obtained in the $k=1$ case to get the total number of signed permutations with no peaks, then the number is $2^{2n-1}$.
\end{proof}


We now show that the results in \cite{sagan} for the symmetric group extend to the case of sign permutations. 


\begin{theorem}\label{thm:maintheorem} Let $S=\{i_1, i_2, \ldots, i_s\}$ be admissible. Then
\[ \#P_B(S,n) = p_B(n)2^{2n-s-1}, \] 
where $p_B(n)=p_B(S,n)$ is a polynomial depending on $S$ such that $p_B(n)$ is an integer for all integral $n$. In addition, the degree of $p_B(n) = i_s - 1$ (when $S = \emptyset$, the degree of $p_B(n) = 0$).

\end{theorem}

\begin{proof}
We prove by induction on $i_1 + i_2 + \cdots + i_s$, following the argument in \cite{sagan}. We previously showed that $\#P_B(\emptyset,n) = 2^{2n-1}$, thus our claim is true for the base case. Our inductive hypothesis is that our claim is true for any set $\hat{S} = \left\{r_1, r_2, \ldots, r_t\right\}$ where $r_1 + r_2 + \ldots + r_t < i_1 + i_2 + \cdots + i_s$. 

We let $k = i_s - 1$. For any $n \geq i_s$, let $\Pi$ be the set of all signed permutations $\pi = \pi_1\pi_2\ldots \pi_n$ such that $P_B(\pi_1\pi_2\ldots \pi_k) = S_1 = S - \{i_s\}$ and $P_B(\pi_{i_s}\ldots \pi_n) = \emptyset$. We can partition $\Pi$ into blocks by the peak set of $\pi$.  In addition to the peaks given by $S_1 = S-\{i_s\}$, there could be a peak at $\pi_{k}$, a peak at $\pi_{i_s}$, or no peak at both $\pi_k$ and $\pi_{i_s}$. Note that these are all the possibilities, and that the three are disjoint. Thus, if we let $S_2 = S_1 \cup \left\{i_s-1\right\}$,  then
\begin{equation}
\#\Pi = \#P_B(S_2,n) + \# P_B(S,n) + \#P_B(S_1,n). 
\end{equation}

First, we find $\#\Pi$. Recall that for $\pi \in \Pi$, $P_B(\pi_1\ldots \pi_k) = S_1$ and $P_B(\pi_{i_s}\ldots \pi_n)$ equals  $\emptyset$. Therefore to construct any $\pi$, first we choose $k = i_s - 1$ elements to be in the first section. For signed permutations, if an integer $m \in [n]$ is in the permutation, then $-m$ cannot be, and vice versa. Therefore we choose $k$ elements from a set of $n$ elements. The number of ways to do so is ${n \choose k}$. Then we create a signed permutation from these $n$ elements, arranged in a way such that their peak set is $S_1$. We have denoted the number of ways to do so by $\#P_B(S_1,k)$. Finally we arrange the last $n-k$ items such that their peak set is $\emptyset$. The number of ways to do this is $\#P_B(\emptyset, n-k)$. Therefore, $\#\Pi = {n \choose k}\#P_B(S_1,k)\#P_B(\emptyset,n-k)$.

By our inductive assumption, $\#P_B(S_1,k) = p_1(k)2^{2k-(s-1)-1} = p_1(k)2^{2k-s}$, where $p_1(k)$ is a polynomial of degree $i_{s-1}-1$, and we know from the base case that $\#P_B(\emptyset,n-k) = 2^{2(n-k)-1} = 2^{2n-2k-1}$. Thus, \[ \#\Pi = p_1(k){n \choose k}2^{2n-s-1}.\]

Similarly, we find that by our inductive assumption, $\#P_B(S_1,n) = p_1(n)2^{2n-s}$ and $\#P_B(S_2,n) = p_2(n)2^{2n-s-1}$, where $p_2(n)$ has degree $i_s-2$.  Thus, 
\begin{align*}
\#P_B(S,n) &= \#\Pi - \#P_B(S_1,n) - \#P_B(S_2,n)\\
&= p_1(k){n \choose k}2^{2n-s-1}-p_1(n)2^{2n-s}-p_2(n)2^{2n-s-1}\\
 &= \left(p_1(k){n \choose k} - 2p_1(n)-p_2(n)\right)2^{2n-s-1}.
\end{align*}

Note that $p_1(k){n \choose k}$ has degree $k$. We know $2p_1(n)$ has degree $i_{s-1}-1 < k$, and $p_2(n)$ has degree $i_s-2 < k$, then the coefficient of $2^{2n-s-1}$ is a polynomial of degree $k = i_s-1$. Since ${n \choose k}$, $2p_1(n)$, and $p_2(n)$ all have integral values at integer $n$, the difference also has integral values at integer $n$.
\end{proof}


The proof of Theorem \ref{thm:maintheorem} immediately yields the following recursive formula for the polynomial $p_B(S,n)$.
\begin{corollary}\label{cor:maincorollary}

If $S \not= \emptyset$ is admissible and $m=\max S$ then \[p_B(S,n)=p_1(m-1) \binom{n}{m-1} -2p_1(n)-p_2(n)\] where $S_1=S-\{m\}, S_2=S_1 \cup \{m-1\}$, and $p_i(n)=p_B(S_i,n)$ for $i=1,2$.

\end{corollary}

\begin{proof}
Looking at the proof of Theorem \ref{thm:maintheorem}, we can see that the greatest element in $S$ is $i_s$ which means that $m=i_s$. We also let $k=i_s-1=m-1$, then by Theorem \ref{thm:maintheorem} we have
\begin{align*}
p_B(S,n)= p_1(k) \binom{n}{k} -2p_1(n)-p_2(n)=p_1(m-1){n \choose m-1} -2p_1(n)-p_2(n).           
\end{align*}
\end{proof}


The recursive formula for the polynomial $p_B(n) = p_B(S,n)$ in terms of $p_B(S_1,k)$, $p_B(S_1,n)$, and $p_B(S_2,n)$, given in Corollary \ref{cor:maincorollary}, is identical to the recursive formula for the polynomial when working in $S_n$, where Theorem 3 in \cite{sagan} shows that $\#P(S,n) = p(n)2^{n-|S|-1}$. In fact, we can show in general that  $p_B(n)= p(n)$.


\begin{lemma}\label{lem:polynomials-same}
For any non-empty $S = \{i_1,i_2,\hdots,i_s\}$,  let $S_2 = \{i_1,i_2,\hdots,i_{s-1}, i_{s}-1\}$. Then $p_B(S,n) + p_B(S_2,n) = p(S,n) + p(S_2,n)$ implies $p_B(S,n) = p(S,n)$ for any $n$.
\end{lemma}

\begin{proof}
We assume that $p_B(S,n) + p_B(S_2,n) = p(S,n) + p(S_2,n)$ and consider two cases $s = 1$ and $s > 1$.

Case $s = 1$: Since $S$ has only one element, we can write $S = \{m\}$ and induct on $m$. The base case is $m = 1$, since we cannot have a peak at position $1$, $p_B(S,n) = p(S,n) = 0$. Now, we assume our claim is true for $S = \{m\}$, where $m\geq 1$. 
Consider $S = \{m+1\}$, our assumption gives us that $p_B( \{m+1\},n) + p_B(\{m\},n) = p( \{m+1\},n) + p(\{m\},n)$, and our inductive hypothesis gives us our result.

Case $s > 1$: In this case, we notice that if we write $i_s = i_{s-1}+z$, then $S = \{i_1, i_2, \hdots, i_{s-1}, \allowbreak i_{s-1}+z\}$ and $S_2 = \{i_1, i_2, \hdots, i_{s-1}, i_{s-1}+(z-1)\}$. Therefore, by induction on $z$ we can easily see that, assuming $p_B(S,n) + p_B(S_2,n) = p(S,n) + p(S_2,n)$ implies $p_B(S,n) = p(S,n)$. 
\end{proof}

\begin{theorem}\label{thm:polynomials-same}
For any $S = \{i_1,i_2,\hdots,i_s\}$ and for any $n$, if $\#P(S,n) = p(n)2^{n-s-1}$ and $\#P_B(S,n) = p_B(n)2^{2n-s-1}$, then $p(n) = p_B(n)$ is the same polynomial in $n$ depending on $S$ for both $S_n$ and $B_n$.
\end{theorem}

\begin{proof}

First, let $S_1 = S - \{i_s\}$ and $S_2 = S_1 \cup \{i_s-1\}$. By Corollary \ref{cor:maincorollary} we can write 
\begin{equation}\label{eqn:polynomial-recursion-Bn}
p_B(S,n) + p_B(S_2,n) = {n \choose i_s-1}p_B(S_1,i_s-1)-2p_B(S_1,n).
\end{equation}
By \cite{sagan}, we can also write 
\begin{equation}\label{eqn:polynomial-recursion-Sn}
p(S,n) + p(S_2,n) = {n \choose i_s-1}p(S_1,i_s-1)-2p(S_1,n).
\end{equation} 

By Lemma \ref{lem:polynomials-same} it is enough to show that the recursive formulas for the polynomials for $S$ of size $s$ are equal, i.e. $p_B(S,n) + p_B(S_2,n) = p(S,n) + p(S_2,n)$. We show this by inducting on $s$.

The base case is $s = 0$. Here, $S_1 = \emptyset$, then Equation (\ref{eqn:polynomial-recursion-Bn}) and Equation (\ref{eqn:polynomial-recursion-Sn}) are in terms of the polynomials for $S = \emptyset$. By  Proposition \ref{prop:nopeakssigned}, we have that $p_B(\emptyset,n) = 1$, and the same is true for $p(\emptyset,n)$ \cite{sagan}. Therefore $p_B(S,n) + p_B(S_2,n) = p(S,n) + p(S_2,n)$ in the base case.

Now, we assume that our claim holds for $s-1\geq 0$. Then if we let $S_3 = (S_1 - \{i_{s-1}\}) \cup \{i_{s-1}-1\}$, our inductive hypothesis gives us that $p_B(S_1,n) + p_B(S_3,n) = p(S_1,n) + p(S_3,n)$. By Lemma \ref{lem:polynomials-same} this hypothesis implies that $p_B(S_1,n) = p(S_1,n)$. By the same argument, since our inductive hypothesis applies for any $n$, $p_B(S_1,i_s-1) = p(S_1,i_s-1)$. Therefore Equation (\ref{eqn:polynomial-recursion-Bn}) and Equation (\ref{eqn:polynomial-recursion-Sn}) will give us our claim for any $s$.

\end{proof}

In \cite{sagan}, the polynomials $p(n)$ have been computed for several special cases of $S$ using Corollary \ref{cor:maincorollary}. Hence using Theorem  \ref{thm:polynomials-same} we obtain the same polynomials for $B_n$ given in the following corollaries. 


\begin{corollary}\label{thm:m}
If $S = \left\{m\right\}$ is admissible then \[ p_B(S,n) = {n-1 \choose m-1} - 1 .\]
\end{corollary}

\begin{corollary}\label{thm:2-m}
If $S = \left\{2,m\right\}$ is admissible then \[ p_B(S,n) = (m-3){n-2 \choose m-1} + (m-2){n-2 \choose m-2} -  {n-2 \choose 1}.\]
\end{corollary}

\begin{corollary}
If $ S = \left\{2,m,m+2\right\}$ is admissible then
\[  p_B(S,n) = m(m-3){n \choose m+1} -2(m-3){n-2 \choose m-1} -2(m-2){n-2 \choose m-2} +2{n-2 \choose 1 }. \]
\end{corollary}


\section{Permutations in $B_n$ with $\pi_0 = 0$}\label{B0n}

Recall that a peak is defined such that the permutation $\pi \in B_n$ has a peak at position $i$ if $\pi_{i-1} < \pi_{i} > \pi_{i+1}$. Therefore if we introduce the assumption that $\pi_0 = 0$ for all $\pi \in B_n$, then a permutation $\pi$ can have a peak at position $1$ if $0 < \pi_1 > \pi_2$, that is if $\pi_1$ is positive and $\pi$ has a descent at 1 (i.e. $\pi_1>\pi_2$). We define $\widehat{P}_B(S,n)$ to be the set of all permutations in $B_n$ with peak set $S$ with the assumption that $\pi_0 = 0$. 
The number of $n$-admissible sets is also given by the Fibonacci sequence.  

\begin{proposition}\label{thm-nadmissible}
Let $A_n$ be the set of $n$-admissible peak sets $S$. Then the size of $A_n$ is given by the $(n+1)$th Fibonacci number. 
\end{proposition}

\begin{proof} The proof is exactly the same as for Proposition \ref{thm-nadmissible0} with initial values  $A_1 = \{\emptyset\}$ and $A_2 = \{\emptyset,\{1\}\}$.
\end{proof}

The next theorem shows that similar recursion holds for computing values for $\#\widehat{P}_B(S,n)$.  We then proceed to compute special cases for various $n$-admissible peak sets $S$.


\begin{theorem} \label{thm:SaganrelationBn0} 
Let $S =\left\{   i_1,i_2,\ldots,  i_s  \right\}$ be an admissible set then,
\[  \#\widehat{P}_B(S,n)={n \choose i_s-1} \#\widehat{P}_B(S_1,i_s-1)2^{2(n-i_s)+1} - \#\widehat{P}_B(S_1,n) - \#\widehat{P}_B(S_2,n).  \]
where $S_1=S-\{i_s\}$ and $S_2=S_1\cup \{i_s-1\}$. 
\end{theorem}

\begin{proof}
This equation is based on the construction given in the proof of Theorem \ref{thm:maintheorem}.  We omit the details as they are exactly the same as in that proof. 
\end{proof}


In the following theorem we show that $\#\widehat{P}(\emptyset,n)$ is given by sequence A007051 in the OEIS.  The first few values of this sequence are $1,Œ2, 5, 14, 41, 122, 365, \ldots.$ 

\begin{theorem}\label{thm:Bn0emptyset}

Let $S = \emptyset$. Then \[ \#\widehat{P}_B(S,n) = \frac{3^n+1}{2}.\]

\end{theorem}

\begin{proof}

Let $\pi \in \widehat{P}_B(\emptyset,n)$. A general shape for $\pi$ is given by Figure 1, where the section labeled A is negative and decreasing, the section labeled B is negative and increasing, and the section labeled C is positive and increasing.  According to these sections, we can partition $\pi$ into sections $\pi = \pi_0 \pi_A\pi_B\pi_C$. In general, up to two of these sections can be empty.  We also assume that $\pi_B$ contains the entire section of the permutation that is negative and ascending, including the minimum of $\pi$.  For example,  if $\pi = 0 - 4 -5 -6 -2 -1  \ 3$, then $\pi_A = -4 -5$, $\pi_B= -6 -2 -1$ and $\pi_C=3$.  

\begin{figure}
\label{figure1}

\begin{tikzpicture}
\draw [ultra thick] (0,0) -- (1,-1);
\draw [ultra thick] (1,-1) -- (3,1);
\draw (0,0) -- (3,0);
\draw (1,-1) -- (1,1);
\draw (2,-1) -- (2,1);
\node at (.5,-1.25) {A};
\node at (1.5,-1.25) {B};
\node at (2.5,-1.25) {C};
\end{tikzpicture}

\caption{General shape for $\pi \in \widehat{P}_B(\emptyset,n)$.}

\end{figure}
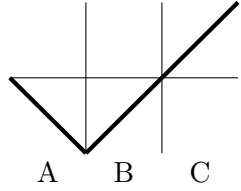

Now, define a function $f$ from $B_n$ to the set of partitions of $[n+1]$ into at most 3 blocks. Let $\pi \in B_n$. If $\pi = \pi_0\pi_A\pi_B\pi_C$, then we let $A$, $B$, and $C$ be the subsets of $[n]$ that correspond to the absolute values of the sections $\pi_A$, $\pi_B$, and $\pi_C$, respectively. Then $f$ maps $\pi$ to the partition of $[n+1]$ into at most 3 blocks, given by $\left\{A,B,C \cup \left\{n+1\right\}\right\}$, where if a section is empty it is not represented in the partition. Then $f(\pi)$ is in the set of partitions of $[n+1]$ into at most 3 blocks.

Next, we define the inverse of $f$ from the set of partitions of $[n+1]$ into at most 3 blocks to $B_n$. Let $P$ be such a partition. We write $P$ as a set of three blocks, where we allow some of the blocks to be empty. I.e., if $P = \left\{P_1,P_2\right\}$, we write $P = \left\{P_1,P_2,\emptyset\right\}$. If $P_j$ is the block containing $n+1$, then we let $C = P_j - \left\{n+1\right\}$. If $P_i$ is the block containing the maximum value of the remaining two blocks, then we let $B = P_i$, and we let $A$ be the remaining block. Hence,  $P$ maps to the signed permutation $\pi = \pi_0\pi_A\pi_B\pi_C$, such that $\pi_A$ is given by negating the elements of $A$ and ordering them so they are decreasing, $\pi_B$ is given by negating the elements of $B$ and ordering them so they are increasing, and $\pi_C$ is given by ordering the elements of $C$ so they are increasing.

It is known that the size of the set of partitions of $[n+1]$ into at most $3$ blocks is given by the first three Stirling numbers of the second kind, $S(n+1,1) + S(n+1,2) + S(n+1,3) = \frac{3^n+1}{2}$. 
Therefore,  the size of $\widehat{P}_B(S,n)$ is $(3^n+1)/2$.
\end{proof}


\subsection{Parity of $\widehat{P}_B(S,n)$}
In the previous section we showed that $P_B(S,n)$ was always a multiple of a power of 2, and hence always even.  This is no longer the case for $\widehat{P}_B(S,n)$ as we show in the next theorem. 
\begin{theorem}
Let $S = \{i_1,i_2,\hdots,i_s\}$. Then $\#\widehat{P}_B(S,n)$ is even if $S$ contains some even number or if $n$ is odd, and is odd otherwise.
\end{theorem}

\begin{proof}
We induct on $i_1+i_2+\cdots+i_s$.

Our base case is $i_1+i_2\cdots+i_s = 0$, where, $S = \emptyset$. Clearly, $S$ contains no even elements. Note that $\#\widehat{P}_B(\emptyset,n) = (3^n+1)/2$ is even if $n$ is odd and odd if $n$ is even, thus our claim holds.

Recall Theorem \ref{thm:SaganrelationBn0}, which states that if $S_1 = S - \{i_s\}$ and $S_2 = S_1 \cup \{i_s-1\}$, then 
\[ \#\widehat{P}_B(S,n) = {n \choose i_s-1}\#\widehat{P}_B(S_1,i_s-1)2^{2(n-i_s)+1}-\#\widehat{P}_B(S_1,n)-\#\widehat{P}_B(S_2,n).\] 
Note that the first term will always be even, since it is multiplied by $2$ with some positive exponent. Therefore $ \#\widehat{P}_B(S,n)$ is even if and only if $\#\widehat{P}_B(S_1,n)+\#\widehat{P}_B(S_2,n)$ is even.  If $n$ is odd, then by our inductive assumption, $\#\widehat{P}_B(S_1,n)$ and $\#\widehat{P}_B(S_2,n)$ are both even, then their sum is even.  

Now, consider the case where $n$ is even.  If $S$ has at least one even element, let $i_j$ be the first even element in $S$. Either $i_j \in S_1$ or $i_j = i_s$. In the first case, our inductive hypothesis implies that $\#\widehat{P}_B(S_1,n)$ and $\#\widehat{P}_B(S_2,n)$ are both even, then their sum is even. In the second case, $S_1$ has no even elements, thus by our inductive hypothesis, $\#\widehat{P}_B(S_1,n)$ is odd. Note that if $i_s$ is even, then $i_s-1$ is odd and $S_2$ has no even elements. Therefore $\#\widehat{P}_B(S_2,n)$ is also odd, thus their sum is even.

Now consider the case where $S$ contains no even elements and $n$ is still even. Since $S_1$ contains no even elements, by our inductive hypothesis $\#\widehat{P}_B(S_1,n)$ is odd. But since $i_s$ is odd, $i_s-1$ must be even. Therefore by our inductive hypothesis $\#\widehat{P}_B(S_2,n)$ is even, hence their sum is odd.
\end{proof}

\subsection{Relationship between $\#P_B(S,n)$ and $\#\widehat{P}_B(S,n)$}  The following relation between $\#P_B(S,n)$ and $\#\widehat{P}_B(S,n)$ allows us to extrapolate some results from Section 2.

\begin{proposition}\label{prop:relation-B_n}
If $S$ is admissible, then
\[\# P_B(S,n) = \#\widehat{P}_B(S,n)+\#\widehat{P}_B( S \cup \{1\},n).\]
\end{proposition}

\begin{proof}

 For any $\pi \in P_B(S,n)$, either $\pi$ has a descent at position $1$ (i.e. $\pi_1 < \pi_2$), or it does not. Therefore we can write $P_B(S,n)$ as a union of disjoint sets $P_B(S,n)=P_{\alpha}(S,n) \cup P_{\beta}(S,n)$ where $\pi \in P_{\alpha} (S,n)$ has a descent at position $1$ and $\pi \in P_{\beta}(S,n)$ does not. 
Note that $\pi \in P_{\alpha} (S,n)$ correspond to an element in  $\widehat{P}_B(S \cup \{1\},n)$ by adding a zero at the beginning of $\pi$. Hence,  $\#P_{\alpha}(S,n) = \#\widehat{P}_B(S \cup \{1\},n)$. Similarly,  any $\pi \in P_{\beta}(S,n)$ corresponds to an element in $\widehat{P}_B(S,n)$ and thus $\#P_{\beta} (S,n) = \#\widehat{P}_B(S, n)$.
Therefore $\#P_B(S,n) = \#\widehat{P}_B(S,n) + \#\widehat{P}_B(S \cup \{1\},n).$
\end{proof}

Proposition \ref{prop:relation-B_n} implies the following corollary.

\begin{corollary}
If $S$ is admissible and $2 \in S$, then
\[\# P_B(S,n)=\#\widehat{P}_B(S,n).\]
\end{corollary}

\begin{proof}
If $2 \in S$, then that means that the peak set $S \cup \{1\}$ has two consecutive peaks which is clearly not possible. This means that $\#\widehat{P}_B( S \cup \{1\},n)=0$, implying that $\# P_B(S,n)=\#\widehat{P}_B(S,n)$ using Proposition \ref{prop:relation-B_n}.
\end{proof}

Using this relation, we are able to find a formula for $\#\widehat{P}_B(S,n)$ in the case where the permutation has one peak.

\begin{proposition}\label{prop:case-m}
Let $S =\{m\}$ be admissible, then
\[\#\widehat{P}_B(\{m\},n) =  4^{n-m-1}\sum_{i=1}^{m}{n \choose m-i}(3^{m-i}+1)4^{i}(-1)^{i+1}  -(m\bmod 2)\left[\frac{3^{n}+1}{2}\right].  \]
\end{proposition}

\begin{proof}
We will induct on $m$. First let $m=1$ then using Theorem \ref{thm:Bn0emptyset} and Proposition \ref{prop:relation-B_n},
\begin{align*}
\#\widehat{P}_B(\{1\},n)  & = \#P_B(\emptyset,n) - \#\widehat{P}_B(\emptyset,n)\\
& = 2^{2n-1}-\left[\frac{3^{n}+1}{2}\right]\\
& =4^{n-2}(2)(4) -\left[\frac{3^{n}+1}{2}\right] \\
& = 4^{n-1-1}\sum_{i=1}^{1}{n \choose 1-i}(3^{1-i}+1)4^{i}(-1)^{i+1} - (1\bmod 2)\left[\frac{3^{n}+1}{2}\right].    
\end{align*}	
We assume our claim is true for $m$ and we consider $m+1$. Apply Theorem \ref{thm:SaganrelationBn0} for the peak set $S= \{m+1\}$
to obtain the following
\[\#\widehat{P}_B (\{m+1\},n) =  {n \choose m}\#\widehat{P}_B(\emptyset,m)\#P_B(\emptyset,n-m) - \#\widehat{P}_B(\emptyset,n) - \#\widehat{P}_B(\{m\},n).\]
Apply Proposition \ref{prop:nopeakssigned},  Theorem \ref{thm:Bn0emptyset} and the inductive hypothesis, then
\begin{align*}
& \#\widehat{P}_B(\{m+1\},n) = {n \choose m}\left( \frac{ 3^{m}+1}{2}\right)2^{2(n-m)-1} - \left( \frac{ 3^{n}+1}{2}\right) - \\  
& \qquad 4^{n-m-1}\sum_{i=1}^{m}\bigg[{n \choose m-i}(3^{m-i}+1) 4^{i}(-1)^{i+1} \bigg]+ (m\bmod2)\left( \frac{ 3^{n}+1}{2}\right)\\
& = -4^{n-m-1} \bigg( {n \choose m-0}(3^{m-0}+1)(4)^{0}(-1)^{0+1} \\
& \qquad+\sum_{i=1}^{m}{n \choose m-i}(3^{m-i}+1)4^{i}(-1)^{i+1}\bigg)  -(m+1\mod2)\left( \frac{ 3^{n}+1}{2}\right)\\
& = -4^{n-m-1}\sum_{i=0}^{m}{n \choose m-i}(3^{m-i}+1)4^{i}(-1)^{i+1}   - (m+1\bmod2)\left( \frac{ 3^{n}+1}{2}\right)\\ 
& =  -4^{n-(m+1)-1}\sum_{i=1}^{m+1}{n \choose (m+1)-i}(3^{(m+1)-i}+1)4^{i}(-1)^{i+1}\\
&\qquad -(m+1\mod2)\left( \frac{ 3^{n}+1}{2}\right). \qedhere
\end{align*}

\end{proof}


Again, applying the relation to the case $S = \left\{m\right\}$ gives us a result for $S = \left\{1,m\right\}$.

\begin{corollary}\label{prop:case-1,m}
Let $S =\{1,m\}$ be admissible, then
\begin{align*}
 \#\widehat{P}_B(S,n) &= 2^{2n-2}\left[   {n-1 \choose m-1} -1  \right]  -  4^{n-m-1}\sum_{i=1}^{m}{n \choose m-i}(3^{m-i}+1)(4)^{i}(-1)^{i+1} \\  
&+ (m\bmod 2)\left[\frac{3^{n}+1}{2}\right].  
\end{align*}
\end{corollary}

\begin{proof}
Apply Proposition \ref{prop:relation-B_n} to obtain the following
\[\# P_B(\{m\},n) = \#\widehat{P}_B(\{m\},n)+\#\widehat{P}_B( \{m\} \cup \{1\},n).\]
We then use Theorem \ref{thm:maintheorem} together with Corollary \ref{thm:m} for $\# P_B(\{m\},n)$ and Proposition \ref{prop:case-m} for $\#\widehat{P}_B(\{m\},n)$. The rest follows.
\end{proof}


The following result is a general result for a two-element peak set $S$.
\begin{proposition}
Let $S = \{m, m + z\}$ be admissible, then $\#\widehat{P}_B(S,n)$ equals

\begin{align*}
&\sum_{i=0}^{z-2}(-1)^{i}2^{2(n-m-z+i+1)-1}{n \choose m+z-1-i}\bigg[4^{z-i-2}\sum_{j=1}^{m}\bigg((3^{m-j}+1)4^{j}(-1)^{j+1}\\
& {m+z-i-1 \choose m-j}\bigg)-(m \mod2)\left(\frac{3^{m+z-i-1}+1}{2}\right)\bigg]-(z-1\mod2)\bigg[4^{n-m-1}\\
&\sum_{i=1}^{m}\bigg((3^{m-i}+1)4^{i}(-1)^{i+1}{n \choose m-i}\bigg) - (m\mod2)\left(\frac{3^{n}+1}{2}\right)\bigg].
\end{align*}

\end{proposition}

\begin{proof}
Let $S = \{m, m + z\}$ be admissible, let $S_1 = \{m\}$ and $S_2 = \{m, m + z-1\}$. Apply Theorem \ref{thm:SaganrelationBn0} to obtain the following recursive formula,
\begin{eqnarray*}
\#\widehat{P}_B(S,n) &=& {n \choose m+z-1}\#\widehat{P}_B(S_1,m+z-1)\#P_B(\emptyset,n-(m+z)+1)\\ 
&&- \#\widehat{P}_B(S_1,n) -  \#\widehat{P}_B(S_2,n).
\end{eqnarray*}

Then apply the recursion $\#\widehat{P}_B(S_2,n)$ until $m+z-1$ approaches $m+1$,  then if $a=m+z-1$ we arrive at the following formula for 
$\#\widehat{P}_B(S,n)$
\begin{eqnarray*}
&&\sum_{i=0}^{z-2}\left[ (-1)^{i}{n \choose a-i}\#\widehat{P}_B(\{m\},m+z-1-i) 
  \#P_B(\emptyset,n-(m+z)+1-i)\right]  \\   && - (z-1\mod2)\#\widehat{P}_B(\{m\},n).
\end{eqnarray*}
Using Proposition \ref{prop:case-m} and Proposition \ref{prop:nopeakssigned} we obtain the result. 
\end{proof}


We have the following special case when $S$ is a three element set. 

\begin {proposition}\label{prop:case-1,m,m+2}
Let $S =\{1,m,m+2\}$ be admissible, then 
\begin{align*}
\#\widehat{P}_B(S,n) &=4^{n-m-1}\sum_{i=1}^{m} (3^{m-i}+1)(4)^{i}(-1)^{i+1}\left[{n \choose m-i} - \frac{1}{2}{n \choose m+1}{m+1 \choose m-i} \right]\\
& +4^{n-1}\left[ \frac{m-1}{2}{n \choose m+1}+1-{n-1 \choose m-1}   \right]\\
& + (m \mod 2)\left[{n \choose m+1}  \left( \frac{3^{m+1}+1}{2}   \right)     2^{2(n-m-1)-1}  -\frac{3^{n}+1}{2}  \right].
\end{align*}

\end{proposition}

\begin{proof}
Let $S_1 = \{1,m \}$ and let $S_2 = \{1,m, m+1\}$ and apply  Theorem \ref{thm:SaganrelationBn0}. Note that $\#\widehat{P}_B(S_2,n) = 0$, then 
\begin{align*}
\#\widehat{P}_B(S,n) = {n \choose m+1}\#\widehat{P}_B(S_1,m+1)\#P_B(\emptyset,n-(m+2)+1) - \#\widehat{P}_B(S_1,n).
\end{align*}
The result follows from Corollary \ref{prop:case-1,m} and and Proposition \ref{prop:nopeakssigned}. 
\end{proof}


\section{Permutations in $S_n$ with $\pi_0 = 0$}\label{S0n}

Let $S_n$ be the set of all  permutations  $\pi = \pi_1 \pi_2 \ldots \pi_n$ of $[n]$. Recall that we define the set $P(\pi) $ as the set of all peaks of $\pi$. Now, we introduce the condition $\pi_0  = 0$, which will allow our peak set to contain $i =1$. Define $\widehat{P}(\pi)$ as the set of all peaks of $\pi$ with  $\pi_0  = 0$, and $\widehat{P}(S,n)$ as the set of all permutations of $S_n$ having $\pi_0  = 0$ and peak set S.


We first give the recursive method that will allow us to compute formulas for the special peak sets $ S =\{\{m\},\{1,m\},\{1,m,m+2\},\{1,m,n-1\}\}$.  This recursive formula is based on Theorem \ref{thm:maintheorem}.

\begin{theorem} \label{thm:Saganrelation} 
Let $S =\left\{   i_1,i_2,\ldots,  i_s  \right\}$ be an admissible set then,
\[  \#\widehat{P}(S,n)={n \choose i_s-1} \#\widehat{P}(S_1,i_s-1)2^{n-i_s}  -  \#\widehat{P}(S_1,n)  -  \#\widehat{P}(S_2,n).  \]
\end{theorem}

\begin{proof}
The method to obtain this expression is based on the construction of $\Pi$ in Theorem \ref{thm:maintheorem}. For this we let the max$(S) = i_s$   and $k = i_s-1$, and the rest follows. 
\end{proof}


We need the next result in order to prove future cases.

\begin{proposition}\label{prop:emptyset-Sn}
The number of permutations with peak set  $\widehat{P}(\pi) = \emptyset $ is equal to one.
\end{proposition}

\begin{proof}
Let $\pi=\pi_0 \pi_1 \hdots \pi_n$ be a permutation in $\widehat{P}(\emptyset, n)$, and $\pi_0 = 0$. Now suppose $\pi_1>1$, then there is an integer $m \in \left\{2,3, \hdots n \right\}$ such that $\pi_m = 1$, hence there is a peak at a position $i \in \left\{ 1, \hdots m-1  \right\}$. Therefore, in order to have no peaks, $\pi_1 $ must equal 1. Now suppose $\pi_2>2$, then there is an integer $m \in \left\{3,4, \hdots n \right\}$ such that $\pi_m = 2$, hence there is a peak at a position $i \in \left\{ 2, \hdots , m-1  \right\}$. Therefore, in order to have no peaks, $\pi_2 $ must equal 2. Apply the same procedure $n - 1$ times. Thus in order to have a permutation of the form $\pi=\pi_0 \pi_1 \hdots \pi_n$  with $\pi_0 = 0$ and peak set $\widehat{P}(\pi) = \emptyset $, the permutation must satisfy $\pi_i = i,$ for  $i \in \left\{ 0,1, \hdots ,n \right\} $. 
\end{proof}


We need the next result as a base case for the special case $S = \{m\}$.

\begin{proposition}\label{prop:[1]-Sn}
Let $S = \left\{1\right\}$. Then \[\widehat{P}(S,n) = 2^{n-1}-1. \]
\end{proposition}

\begin{proof}

Let $\pi \in P( \emptyset , n)$. Then either $\pi \in \widehat{P}( \left\{1 \right\} , n)$ or by Proposition \ref{prop:emptyset-Sn}, $\pi$ is the identity in $S_n$. Therefore by \cite{sagan} Proposition 2 we have 
$\#\widehat{P}( \left\{1 \right\}, n) = \#P( \emptyset , n) -1 
= 2^{n-1}-1. \qedhere $

\end{proof}
Now we make use of the recursive formula in Theorem \ref{thm:Saganrelation} and the result in Proposition \ref{prop:emptyset-Sn} to obtain a recursive formula for the case when $S= \{m\}$. This will lead us to have a closed formula for this case.

\begin{lemma}\label{lemma:recursiveformula}
Let $S = \left\{m\right\}$ be admissible. Then we can find the number of permutations with peak set $S$ recursively by \[ \#\widehat{P}(S,n) = 2^{n-m}{n \choose m-1} - \#\widehat{P}(\left\{m-1\right\},n) -1.\]
\end{lemma}

\begin{proof}

Let $\Pi$ be the set of all $\pi \in S_n$ such that if $\pi = \pi_1 \pi_2 \hdots \pi_{m-1} \pi_m \hdots \pi_n$, then $\widehat{P}(\pi_1\pi_2 \hdots \pi_{m-1}) = \emptyset$ and $P(\pi_m \hdots \pi_n) = \emptyset$.

Note that since $\pi$ could have a peak at position $m-1$, position $m$, or neither, then $\Pi$ is a union of disjoint sets $\Pi = \widehat{P}(\{m-1\},n) \cup \widehat{P}(\{m\},n) \cup \widehat{P}(\emptyset,n)$.  Thus,
\begin{equation} \label{eq1}
\#\widehat{P}( \{ m \} ,n) = \#\Pi - \#\widehat{P}(\emptyset,n) - \#\widehat{P}(\{m-1\},n).
\end{equation}

We can construct $\Pi$ by first choosing the $m-1$ first elements, arranging them so their peak set is the empty set, and arranging the $n-m+1$ other elements so their peak set is the empty set. Therefore by Proposition \ref{prop:emptyset-Sn} and by \cite{sagan} Proposition 2 we have
\begin{align*}
 \#\Pi &= {n \choose m-1 }\#\widehat{P}(\emptyset,m-1)\#P(\emptyset,n-m+1)\\
&= {n \choose m-1}2^{n-m}.
\end{align*}

Also, by Proposition \ref{prop:emptyset-Sn}, $\#\widehat{P}(\emptyset,n) = 1$. Therefore by (\ref{eq1}) the result follows.

\end{proof}
\begin{proposition}\label{closed-m}
Let $S = \left\{m\right\}$ be admissible. Then 
\[\# \widehat{P}(S,n) = \sum_{i=1}^m2^{n-i}{n \choose i-1}(-1)^{m-i} -(m\mod2).\]
\end{proposition}

\begin{proof}
We induct on $m$. Our base case is $m = 1$. Then by Proposition \ref{prop:[1]-Sn}, our claim is true. Now, by our inductive assumption, 
\[ \widehat{P}(\left\{m-1\right\},n) = \sum_{i=1}^{m-1}2^{n-i}{n \choose i-1}(-1)^{m-1-i}-(m-1\mod2).\]

Using this value for $\widehat{P}(\left\{m-1\right\},n)$ in the recursive formula given in Lemma \ref{lemma:recursiveformula}, we find
\begin{align*}
\#\widehat{P}(S,n) &= 2^{n-m}{n \choose m-1} - 1 - \left( \sum_{i=1}^{m-1}2^{n-i}{n \choose i-1}(-1)^{m-1-i}-(m-1\mod2)\right)\\
&=  2^{n-m}{n \choose m-1}(-1)^0 + \sum_{i=1}^{m-1}2^{n-i}{n \choose i-1}(-1)^{m-i}-(m\mod2)\\
&= \sum_{i=1}^m2^{n-i}{n \choose i-1}(-1)^{m-i}-(m\mod2).\qedhere
\end{align*}

\end{proof}


From Proposition \ref{closed-m}  we notice that we can factor a power of two out of the summation, in this way we obtain a new formula for the case $S = \{m\}$.

\begin{proposition}\label{close-m-polynomial}
Let $S = \{m\}$ be admissible, then
\begin{equation} \label{eq2}
\#\widehat{P} (\{m\}, n) = \frac{p_{m-1}(n)2^{n-m}}{(m-1)!}-(m\mod2)
\end{equation}
where $p_{m-1}(n) = p(S,n)$ is a polynomial depending on $S$ such that $p(n)$ is an integer for all integral $n$. Also, $deg(p_{m-1}(n)) = m-1$.
\end{proposition}

\begin{proof}
We will prove this by induction on $m$. The case where $m=1$ is true since we already found that $\#\widehat{P}(\{1\},n)=2^{n-1}-1$ where $p(n)=1$ and is of degree $1-1=0$. We will assume that the proposition is true for $m$ and we will prove it for the $m+1$ case. We will first use the recurrence relation in Lemma \ref{lemma:recursiveformula}  with our inductive assumption to get 

\begin{align*}
\#\widehat{P} (\{m+1\}, n) &= \binom{n}{m}2^{n-(m+1)} -1-\left(\frac{p_m(n)2^{n-m}}{(m-1)!}-(m\mod 2)\right)\\
&=\frac{2^{n-(m+1)}}{m!} \left(\frac{n!}{(n-m)!} - 2 m p_m(n)\right)-(m+1\mod 2).
\end{align*}

We now have to prove that \[\frac{n!}{(n-m)!} - 2 m p_m(n)\] is a polynomial in terms of $n$ with degree $m$. Because $m$ is fixed, we can see that this expression is a polynomial in terms of $n$ and we can also clearly see that $n!/(n-m)!$ has degree $m$ and because of the inductive hypothesis $2 m \cdot p_{m}(n)$ has degree $m-1$ which means that the whole expression has degree $m$ which completes the induction. 
\end{proof}
\begin{corollary}
Additionally,
\begin{equation} \label{eq3}
p_m(n)=m! \sum_{i=0}^m 2^i (-1)^i \binom{n}{m-i}.
\end{equation}
\end{corollary}

\begin{proof}
From Proposition \ref{close-m-polynomial} we have a recursive formula for $p_m(n)$,
\[p_m(n)=\frac{n!}{(n-m)!} - 2 m p_{m-1}(n).\]
We will prove this by induction on $m$. The case where $m=0$ is obviously true since using the formula we get that $p_0(n)=1$ which agrees with $\#\widehat{P}(\{1\},n)=2^{n-1} \cdot 1 -1.$ We will assume that the proposition is true for $m$ and we will it prove it for the $m+1$ case. Using the recursive formula and the inductive hypothesis we get
\begin{align*}
p_{m+1}(n)&=\frac{n!}{(n-(m+1))!} - 2 (m+1) \cdot p_{m}(n)\\
&=\frac{n!}{(n-(m+1))!} - 2 (m+1) \left(\frac{n!}{(n-m)!}+m!\sum_{i=1}^m 2^i (-1)^i \binom{n}{m-i}\right) \\
&=\frac{n!}{(n-(m+1))!}-2(m+1)\left(m! \binom{n}{m}+m!\sum_{i=1}^m 2^i (-1)^i \binom{n}{m-i}\right) \\
&=\frac{n!}{(n-(m+1))!}+(m+1)! \sum_{i=1}^{m+1} 2^i (-1)^i \binom{n}{m+1-i}\\
&=(m+1)! \sum_{i=0}^{m+1} 2^i (-1)^i \binom{n}{m+1-i}.
\end{align*}
which is what we wanted thus completing the induction.
\end{proof}

In the following proposition we compute $\#\widehat{P}(\{m\}, n)$ using a different approach.   The new formula we obtain will help us to compute other special cases such  when $S=\{1,n-1\}$ in a simpler way. 

\begin{proposition}\label{alex's-formula}
Let $S=\{n-m\}$ be admissible. Then
\[\# \widehat{P}(S,n)=\sum_{i=0}^{m-1} 2^i \binom{n-(m-i)}{i+1}.\]
\end{proposition}

\begin{proof}
Let $\pi=\pi_1 \pi_2 \hdots \pi_n$ be a permutation in $S_n$. We will prove this proposition by induction on $m$. We will first prove the base case, when $m=1$. Letting $m=1$ means that we will have a peak only in the $(n-1)$-th position. Note that  $\pi_i \in [n]$ which means that the number on the $n-1$ position has to be $n$ because otherwise there would either be no peaks or more than one peaks.  We know that the numbers before the $n-1$ position must be in increasing order, thus the permutation is completely determined by the element in the $n$th position.  There are $\binom{n-1}{1}$ ways to choose the last element.

Now assume the proposition is true for $m\geq 1$,  we will prove that it is true for $m+1$. This means that we have a peak at the $n-(m+1)$-th position, then using reasoning similar to the one used in the inductive hypothesis, we know that $n$ is either in position $n-(m+1)$  or in the last position.   If $n$ is in the position of the peak, the number of permutations that satisfy this condition is equal to the number of ways to choose the last $m+1$ numbers in the permutation times the number of ways to arrange these $m+1$ numbers so that they do not form a peak. This number is equal  $2^{m} \binom{n-1}{m+1}$.

If   $n$ is in  the $n$-th position of the permutation,  then we can reduce the computation to the $m$-th case of the induction.  Thus, 
\[  2^{m} {n-1 \choose m+1} + \sum_{i=0}^{m-1} 2^i \binom{n-(m-i)}{i+1} =\sum_{i=0}^m 2^i \binom{n-(m+1-i)}{i+1}\] 
which was what we were looking for.
\end{proof}


Note that doing a change of variable in the previous result will lead us to obtain better results for the case $S = \{m\}$. 

\begin{remark}\label{remark-mcase}
Let $S=\{m\}$ be admissible. Notice from the Proposition \ref{alex's-formula} that we can write $\# \widehat{P}(S,n)$, as 
\[ \# \widehat{P}(S,n) =   \sum_{i=0}^{n-(m+1)} 2^i \binom{m+i}{i+1}.\]
\end{remark}


\begin{proposition}\label{prop:1-m}
Let $S=\{1,m\}$ be admissible, then
\begin{eqnarray*}
\#\widehat{P}(\{1,m\},n)& =& \sum_{i=1}^{m-2} \binom{n}{m-i} (2^{m-i-1}-1)(2^{n-(m-i+1)})(-1)^{i+1}\\ && -(m \mod2)(2^{n-1}-1).
\end{eqnarray*}
\end{proposition}

\begin{proof}

Let $S=\{1,m\}$ be admissible and let $S_1=\{1\}$   and $S_2=\{1,m-1\}$. Recall Theorem  \ref{thm:Saganrelation} provides the following recursive fomula
\begin{align*}
  \#\widehat{P}(S,n)
= &{n \choose m-1} \#\widehat{P}(\{1\},m-1)\#P(\emptyset,n-m+1)  -  \#\widehat{P}(\{1\},n)\\ &  - \#\widehat{P}(\{1,m-1\},n)\\
= &{n \choose m-1} (2^{m-2}-1)(2^{n-m})  -  (2^{n-1}-1)  -  \#\widehat{P}(\{1,m-1\},n).
\end{align*}
To obtain the terms $(2^{m-2}-1)$ and $(2^{n-1}-1) $ apply Proposition \ref{prop:[1]-Sn} and the term $(2^{n-m+1-1})$ follows from Proposition 2 in  \cite{sagan}. The result follows by induction. 
\end{proof}

\begin{proposition}
Let $S=\{1,m,m+2\}$ be admissible, then  $\#\widehat{P}(S,n)$ equals
\begin{eqnarray*}
&&\sum_{i=1}^{m-2}(2^{m-i-1}-1)(-1)^{i+1}(2^{n-m+i-1})\bigg[\binom{n}{m+1} \binom{m+1}{m-i} \frac{1}{2}- \binom{n}{m-i}\bigg]\\ 
& & + (m \mod2)\left(2^{n-1}-1-\binom{n}{m+1} 2^{n-m-2} (2^m-1)\right).
\end{eqnarray*}
\end{proposition}

\begin{proof}
We apply the same method as in  Proposition \ref{prop:1-m}. For this we let $S_1 = \{1, m\}$ and  $S_2 = \{1, m, m+1\}$.  Note that $\#\widehat{P}(\{1,m,m+1\},n) = 0$ since we can not have consecutive peaks. Then we construct $\Pi$  based on the number of ways to arrange the permutations in $S_1$ and $S_2$  and use Theorem \ref{thm:Saganrelation} to obtain the recursive formula.
\begin{align*}
\#\widehat{P}(\{1,&m,m+2\}, n)= \Pi - \#\widehat{P}(\{1,m\},n) \\
&={n \choose m+1}\#\widehat{P}( \{1,m \}, m+1)\#{P}( \emptyset,n-(m+2)+1) - \#\widehat{P}(\{1,m\},n) 
\end{align*}
Now for the terms $\#\widehat{P}(\{1,m\},m+1)$ and $\#\widehat{P}(\{1,m\},n)$  apply Proposition \ref{prop:1-m}, and for $\#{P}( \emptyset,n-(m+2)+1)$ apply Proposition 2 in \cite{sagan}. The result follows.
\end{proof}
\begin{proposition}
Let $S = \{1,m,n-1\}$ be admissible, then we have the following recursive formula for $\# \widehat{P} ( S,n)$
\begin{eqnarray*}
&&\sum_{i=1}^{m-2}(2^{m-i-1}-1)(2^{n-m+i-1})(-1)^{i+1}\bigg(\frac{1}{2}{n \choose n-2}{n-2 \choose m-i}-{n \choose m-i}\bigg)\\
&&- (m\mod \ 2) \left( (2^{n-2}-2){n \choose 2}-2^{n-1}+1 \right)- \#\widehat{P}( \left \{1,m,n-2 \right \},n).
\end{eqnarray*}

\end{proposition}

\begin{proof}
let $S_1 = \{1, m\}$ and  $S_2 = \{1, m, n-2\}$. Apply Theorem \ref{thm:Saganrelation} to obtain the recursive formula. 
\begin{eqnarray*}
\# \widehat{P} ( \{ 1,m,n-1 \},n) &= & 2{n \choose 2} \#\widehat{P}( \{1,m \},n-2)
\#P( \emptyset,n-(n-1)+1) \\ && -  \#\widehat{P}(\{1,m\},n) -  \#\widehat{P}(\{1,m, n-2\},n). 
\end{eqnarray*}

Now for $\#\widehat{P}( \{1,m \},n-2)$ and $\#\widehat{P}(\{1,m\},n)$ apply Proposition \ref{prop:1-m}. The result follows.
\end{proof}

\subsection{Relationship between $\#P(S,n)$ and $\#\widehat{P}(S,n)$}  In this section we use the relationship between $\#P(S,n)$ and $\#\widehat{P}(S,n)$ to find new formulas for special cases.  We begin by giving this relationship.

\begin{proposition}\label{thm:relation0}
If $S$ is admissible then \[\#P(S,n) = \#\widehat{P}(S,n) + \#\widehat{P}(S\cup\left\{1\right\},n).\]
\end{proposition}

We omit the proof of this proposition since it is identical at the proof of Proposition \ref{prop:relation-B_n}.

\begin{corollary}\label{cor:identity-m-1m}
Let $S=\left\{1,m\right\}$ be an admissible set, then 
\[ \#\widehat{P}(\left\{1,m\right\},n) = \left({n-1\choose m-1}-1\right)2^{n-2} - \#\widehat{P}(\left\{m\right\},n) .\]
\end{corollary}

\begin{proof}
We apply Proposition \ref{thm:relation0} to the case $S = \left\{m\right\}$, using the fact that $\#{P}(\left\{m\right\},n) =  \left({n-1\choose m-1}-1\right)2^{n-2}$ from Theorem 6 in \cite{sagan}.
\end{proof}
\begin{corollary}
Let $S= \left\{1, n-1\right\}$ be an admissible set then,
\[  \#\widehat{P}(S, n) = 2^{n-2}(n-2) - (n-1). \]
\end{corollary}

\begin{proof}
Note that  $S= \left\{1, n-1\right\}$ is a special case of $S= \left\{1, m\right\}$ with $m=n-1$. Now apply Corollary \ref{cor:identity-m-1m} and Remark \ref{remark-mcase}. Then we have, 
\begin{align*}
\#\widehat{P}(\left\{1,n-1\right\},n) &= \left({n-1\choose (n-1)-1}-1\right)2^{n-2} -  \sum_{i=0}^{n-((n-1)+1)} 2^i \binom{(n-1)+i}{i+1}\\
&= 2^{n-2}(n-2)-(n-1).\qedhere
\end{align*}

\end{proof}

We propose now a result that also comes as a result of Proposition \ref{thm:relation0}.

\begin{corollary}\label{cor:identity-2case}
Let $S = \left\{i_1,i_2,\ldots,i_s\right\}$ where $i_1 = 2$. Then \[\#\widehat{P}(S,n) = p(n)2^{n-s-1}\] where $p(n) = p(S,n)$ is an polynomial depending on $S$ with degree $i_s-1$, such that $p(n)$ is an integer for all integral $n$.

Furthermore, if we let $m = max(S)$, $S_1 = S - \left\{m\right\}$, and $S_2 = S_1 \cup \left\{m-1\right\}$, then \[p(S,n) = p(S_1,m-1){n \choose m-1} - 2p(S_1,n)-p(S_2,n).\]
\end{corollary}

\begin{proof}
We apply Proposition \ref{thm:relation0}, and note that $\#\widehat{P}(S \cup \left\{1\right\},n) = 0$ if $2 \in S$. Thus, \[\#\widehat{P}(S,n) = \#{P}(S,n).\] 
Hence, we can apply Theorem 3 in \cite{sagan} for $\#P(S,n)$ to this special case of $\#\widehat{P}(S,n)$.
\end{proof}


\subsection{Parity of $\#\widehat{P}(S,n)$}   Notice from the previous results how the number of permutations varies according 
 to the parity of some integer related to the peaks. This lead us to establish a relation between the parity of the numbers of permutations with a given peak set and the parity of the numbers in the peak set. 

\begin{theorem}
Let $S = \left\{i_1,i_2,\ldots,i_s\right\}$ be admissible. Then $\#\widehat{P}(S,n)$ is even if and only if  there exists $i_j \in S$ such that $i_j$ is even.
\end{theorem}

\begin{proof}

We induct on $i_1+i_2+\cdots+i_s$.  Our base case is $i_1+i_2\cdots+i_s = 0$, where, $S = \emptyset$. Clearly, $S$ contains no even elements. By Proposition \ref{prop:emptyset-Sn}, $\#P(\emptyset,n) = 1$ for all $n$, thus our claim holds.

Now, by Theorem \ref{thm:Saganrelation} if we let $S_1 = S - \left\{i_s\right\}$ and $S_2 = S_1 \cup \left\{i_s - 1\right\}$, then \[ \#\widehat{P}(S,n) = {n \choose i_s-1}2^{n-i_s}\#\widehat{P}(S_1,i_s-1) - \#\widehat{P}(S_1,n) - \#\widehat{P}(S_2,n). \]
$S$ admissible implies $i_s < n$.  Since $n-i_s > 0$, then ${n \choose i_s-1}2^{n-i_s}\#\widehat{P}(S_1,i_s-1)$ is even in all cases.

Assume that $\#\widehat{P}(S,n)$ is even. Then either $\#\widehat{P}(S_1,n)$ and $\#\widehat{P}(S_2,n)$ are both even, or they are both odd. If $\#\widehat{P}(S_1,n)$ is even, then, by the inductive hypothesis, $S_1$ contains some even element. Since $S_1 \subset S$, $S$ then contains some even element. If $\#\widehat{P}(S_2,n)$ is odd, then, by the inductive hypothesis, all its elements are odd, including $i_s - 1$. Therefore $i_s$ is even, hence $S$ contains some even element.

Assume that $S$ contains some even element. Then either $S_1$ contains some even element and $i_s$ is even, or $S_1$ contains some even element and $i_s$ is odd, or $S_1$ does not contain any even elements and $i_s$ is even. In the first case and the second case, by inductive hypothesis $\#\widehat{P}(S_1,n)$ and $\#\widehat{P}(S_2,n)$ are both even, then $\#\widehat{P}(S,n)$ is even. In the third case, $i_s-1$ is odd, thus by the inductive hypothesis $\#\widehat{P}(S_1,n)$ and $\#\widehat{P}(S_2,n)$ are both odd, hence $\#\widehat{P}(S,n)$ is even.

\end{proof}

\section{Acknowledgements}
We thank Ivelisse Rubio, and MSRI for their help and support.  This work was conducted at  MSRI-UP and supported by the National Security Agency (NSA) grant H-98230-13-1-0262 and the National Science Foundation (NSF) grant 1156499.



\begin{thebibliography}{99}

\bibitem{aguiarbergeronnyman} M. Aguiar, N. Bergeron and K.  Nyman,  The peak algebra and the descent algebras of types B and D.  Aguiar, \emph{Trans. Amer. Math. Soc.} {\bf 356} (2004), no. 7, 2781Ð2824.

\bibitem{aguiarnymanorellana} M. Aguiar,  K. Nyman  and R. Orellana,    New results on the peak algebra. \emph{J. Algebraic Combin.}  {\bf 23} (2006), no. 2, 149Ð188.

\bibitem{bergeron} N. Bergeron and C. Hohlweg, Colored peak algebras and Hopf algebras. \emph{J. Algebraic Combin.}  {\bf 24}, no. 3, 299-330.

\bibitem{sagan} S. Billey, K.  Burdzy and B. Sagan, Permutations with given peak set.  \emph{J. of Integer Seq.}, Vol. 16 (2013) Article 13.6.1, 18 pages. 

\bibitem{billeysagan2} S. Billey, K. Burdzy, S. Pal and B. Sagan,  On meteors, earthworms and WIMPs.  Preprint available at http://www.mth.msu.edu/~sagan/Papers/Old/mew.pdf

\bibitem{kasraoui} A. Kasraoui, The most frequent peak set in a random permutation.  Preprint available at http://arxiv.org/abs/1210.5869.

\bibitem{nyman} K. Nyman, The Peak Algebra of the Symmetric Group. \emph{J. Algebraic Combin.} {\bf 17}  (2003) 309-322.

\bibitem{petersen}  T. K. Petersen,  Enriched P-partitions and peak algebras. \emph{Adv. Math.} {\bf 209} (2007), no. 2, 561Ð610. 

\bibitem{stembridge} J. Stembridge,  Enriched P-partitions, \emph{Trans. Amer. Math. Soc.} 349 (1997), no. 2, 763-788.

\bibitem{strehl} V. Strehl, Enumeration of alternating permutations according to peak sets, \emph{J. Combin. Theory Ser. A} {\bf 24} (1978) 238-240.

\bibitem{warren} D. Warren and E. Seneta, Peaks and Eulerian numbers in a random sequence, \emph{J. Appl. Probab.} {\bf 33} (1996) 101-114. 
\end{thebibliography}
\end{document}